\documentclass[11pt]{amsart}
\usepackage{amsmath,amssymb,amsthm,amsfonts,bm}
\usepackage{geometry}
\usepackage[usenames,dvipsnames]{xcolor}
\usepackage{hyperref}
\hypersetup{colorlinks=true, citecolor=red, linkcolor=blue, filecolor=magenta, urlcolor=red}
\usepackage[capitalize]{cleveref}
\usepackage{comment}

\theoremstyle{plain}
\newtheorem{thm}{Theorem}[section]
\newtheorem{prop}[thm]{Proposition}
\newtheorem{lem}[thm]{Lemma}

\newtheorem{claim}[thm]{Claim}

\theoremstyle{definition}
\newtheorem{defn}[thm]{Definition}
\newtheorem{rem}[thm]{Remark}

\numberwithin{equation}{section}

\newcommand{\C}{\mathbb{C}}

\newcommand{\Z}{\mathbb{Z}}
\newcommand{\Center}{\operatorname{center}}

\newcommand{\cF}{\mathcal{F}}
\newcommand{\Ff}{\mathcal{F}}
\newcommand{\cG}{\mathcal{G}}
\newcommand{\Gg}{\mathcal{G}}

\newcommand{\Ii}{\mathcal{I}}

\newcommand{\cH}{\mathcal{H}}
\newcommand{\cJ}{\mathcal{J}}
\newcommand{\cO}{\mathcal{O}}
\newcommand{\Lie}{\operatorname{Lie}}
\newcommand{\mld}{\operatorname{mld}}
\newcommand{\alg}{\operatorname{alg}}
\newcommand{\Gal}{\operatorname{Gal}}
\newcommand{\Stab}{\operatorname{Stab}}

\newcommand{\Aut}{\operatorname{Aut}}

\newcommand{\rk}{\operatorname{rank}}

\newcommand{\Wprod}{W_{\mathrm{prod}}}
\newcommand{\Aprod}{A_{\mathrm{prod}}}
\newcommand{\Uprod}{U_{\mathrm{prod}}}

\title{Shokurov's global index conjecture for threefold foliations}
\author{Jihao Liu, Sheng Qin}
\address{Department of Mathematics, Peking University, No. 5 Yiheyuan Road, Haidian District, Beijing 100871, China}
\address{Beijing International Center for Mathematical Research, Peking University, No. 5 Yiheyuan Road, Haidian District, Beijing 100871, China}
\email{liujihao@math.pku.edu.cn}
\email{qinsheng@stu.pku.edu.cn}
\subjclass[2020]{14E30, 14J30, 37F75}
\keywords{Foliations, global index, Calabi--Yau foliations, klt threefolds}
\date{\today}

\begin{document}

\begin{abstract}
We prove Shokurov's global index conjecture for foliations in dimension at most three. This answers a question of the first author, Meng, and Xie in dimension three. The main result of this paper is partially obtained by generative AI, particularly the Rethlas system.
\end{abstract}

\maketitle
\tableofcontents

\section{Introduction}

We work over the field of complex numbers $\mathbb C$.

The goal of this paper is to prove the following theorem:

\begin{thm}[Shokurov's global index conjecture for threefold foliations]\label{thm:main}
There exists a positive integer $I$ satisfying the following. Let $X$ be a normal projective variety of dimension $\leq 3$ and $\Ff$ a foliation on $X$ with log canonical singularities such that $K_{\Ff}\equiv 0$. Then $IK_{\Ff}\sim 0$.
\end{thm}

Theorem \ref{thm:main} has the following log version.

\begin{thm}[Shokurov's global index conjecture for threefold foliations, log version]\label{thm:main-pair}
Let $\Ii\subset [0,1]\cap\mathbb Q$ be a DCC set. Then there exists a positive integer $I$ depending only on $\Ii$ satisfying the following. Let $(X,\Ff,B)$ be an lc projective foliated triple of dimension $\leq 3$ such that 
$$K_{\Ff}+B\equiv 0\quad \text{and}\quad B\in\Ii.$$
Then $I(K_{\Ff}+B)\sim 0$.
\end{thm}

Theorem \ref{thm:main-pair} answers Shokurov's global index conjecture for foliations \cite[Conjecture 5.3]{LMX24} in dimension $3$. Previously, Theorem \ref{thm:main-pair} was known in the following cases:
\begin{enumerate}
\item When $\Ff=TX$. This corresponds to the classical Shokurov's global index conjecture in dimension $\leq 3$ and was proven in \cite[Theorem 1.8.7]{Xu20} based on \cite[Corollary 1.7]{Jia21}.
\item When $B\not=0$ \cite[Theorem 1.1]{LLM23}.
\item When $\dim X\leq 2$ \cite[Theorem 1.4]{LLM23} based on \cite[Theorem 1]{Per05}.
\end{enumerate}
It is also worth mentioning that the non-effective version of Theorems \ref{thm:main} and \ref{thm:main-pair} was proven in dimension $\leq 3$ or when $\rk\Ff=1$ in \cite[Theorem 12.1]{CS21} and \cite[Theorem 1.2]{CS26}, and for algebraically integrable foliations in \cite[Theorem 1.4]{DLM23}.

\begin{rem}
By following the proofs of the main results in this paper, we actually have the stronger criterion: Under the conditions of Theorem \ref{thm:main}, if $\Ff$ is canonical and not algebraically integrable, then we may take $I\leq 30$. It is, therefore, interesting to classify non-algebraically integrable canonical Calabi--Yau foliations in dimension $\leq 3$ based on their indices.

Another interesting question to ask is whether the boundedness of complements conjecture for foliations \cite[Conjecture 5.1]{LMX24} holds in dimension $3$, which is known to imply Theorem \ref{thm:main-pair} immediately \cite[Proposition 5.4]{LMX24}. Yet this is a much harder question and is only partially known in dimension $2$ for Fano foliations \cite[Proposition 4.1]{CJV24}.
\end{rem}

\subsection*{Use of generative AI and some ideas of the proof}
A substantial part of the proof was developed with the assistance of the Rethlas system; see \cite{Ju+26} for an introduction to the system. The final argument, however, required significant human input, both in supplying missing references and in reorganizing the output into a rigorous proof.

The algebraically integrable case was already well understood by the first author: in that setting, $K_{\Ff}+B$ is crepant to the moduli part of a Calabi--Yau fibration (cf. \cite[Proposition 3.5]{ACSS21}), and the relevant torsion and effective base-point-freeness statements are available in low dimensions. The main difficulty lies in the non-algebraically integrable case, especially in the purely transcendental case.

When the problem was submitted to Rethlas, the system naturally separated the argument according to the rank, the algebraic rank, and the dimension. It quickly found the reduction to dimension three via \cite{Per05}, but it did not initially locate \cite{Jia21,Xu20}. It also identified the Lie-theoretic approach to the rank one non-algebraically integrable case and then to the rank two case with algebraic rank one. After several prompts, the same ideas were adapted to the rank two purely transcendental case. At that stage the canonical case without boundary was essentially settled.

The system was less effective in the algebraically integrable case. It did not find the relevant references \cite{ACSS21,CHLX23,LMX25} until they were supplied by the first author. This reflects a limitation of the Lie-theoretic strategy: it detects algebraic subfoliations from invariant subgroups of automorphism groups, and is therefore most useful when non-algebraic integrability imposes strong restrictions on the group. From the viewpoint of the minimal model program, by contrast, the algebraically integrable case is the more standard one.

Several revisions were needed after the initial output. First, Rethlas treated only the canonical case without boundary; the reduction from the log canonical case is carried out in Proposition~\ref{prop:reduction}, where the ACC for minimal log discrepancies on surfaces \cite{Ale93} is used to close a small gap. Second, the original treatment of the rank two purely transcendental case repeatedly appealed to the rank two algebraic rank one case without recording the necessary finite-cover and equivariance arguments. These details are supplied below in the Lie-theoretic preliminaries and in Sections~\ref{sec:case-d} and~\ref{sec:case-e}.

Due to the limitation of generative AI, it is possible that we have missed some related references in the literature, and we welcome any comments from experts.

\subsection*{Acknowledgements}
The authors were partially supported by the National Key R\&D Program of China \#\allowbreak 2024YFA1014400. The authors would like to thank the Rethlas team, namely Haocheng Ju, Jiedong Jiang, Shurui Liu, Guoxiong Gao, Yuefeng Wang, Zeming Sun, Bin Wu, Liang Xiao, and Bin Dong, for their contributions to the development of Rethlas and its customized version used for the problem studied in this paper. The authors would like to thank Ruochuan Liu and Gang Tian for constant support and encouragement.

\section{Preliminaries}\label{sec:preliminaries}

We adopt the standard notation and terminology for the minimal model program from \cite{Sho92,KM98,BCHM10} and use them freely. We shall only consider foliations in dimension $\leq 3$ in this paper and we shall follow the definitions and notations as in \cite{CS21,CS25} for foliations.  

Throughout the paper, for a normal variety $X$ we write $TX$ for the tangent sheaf $(\Omega_X^1)^\vee$; when $X$ is smooth, this is the usual tangent bundle. We do not use $T_X$ as it may have notation conflicts with the torus in Lie theory, which we shall also use in this paper.

\begin{defn}
Let $X$ be a normal quasi-projective variety. A \emph{foliation} $\Ff$ on $X$ is a saturated subsheaf $\Ff$ of $TX$ that is closed under the Lie bracket. The \emph{rank} of $\Ff$ is the rank of $\Ff$ as a sheaf and is denoted by $\rk\Ff$. The \emph{algebraic rank} of $\Ff$ is the maximal dimension of a variety $Z$ such that there exists a dominant rational map $f: X\dashrightarrow Z$ and a foliation $\Ff_Z$ on $Z$ such that $\Ff=f^{-1}\Ff_Z$, and is denoted by $\rk_{\alg}\Ff$. 

We say that $\Ff$ is \emph{algebraically integrable} if $\rk_{\alg}\Ff=\rk\Ff$. We say that $\Ff$ is \emph{purely transcendental} if $\rk_{\alg}\Ff=0$.
\end{defn}

\begin{defn}
A \emph{foliated triple} $(X,\Ff,B)$ consists of a normal quasi-projective variety $X$, a foliation $\Ff$ on $X$, and an $\mathbb R$-divisor $B$ on $X$ such that $K_{\Ff}+B$ is $\mathbb R$-Cartier.
\end{defn}

 We will frequently use the following result on abundance for log Calabi--Yau foliated triples in dimension $\leq 3$ without citing:

\begin{thm}[{\cite[Theorem 1.7]{CS21},\cite[Theorem 1.2]{CS26},\cite[Theorem 1.8]{LLM23}}]
Let $(X,\Ff,B)$ be a projective foliated triple such that $\dim X\leq 3$ and $K_{\Ff}+B\equiv 0$. Then $K_{\Ff}+B\sim_{\mathbb R}0$.
\end{thm}

\begin{lem}\label{lem:R-to-Q-linear}
Let $D$ be a $\mathbb Q$-Cartier $\mathbb Q$-divisor on a normal variety $X$. If $D\sim_{\mathbb R}0$, then $D\sim_{\mathbb Q}0$.
\end{lem}
\begin{proof}
Write $D=\sum_{i=1}^m a_i\operatorname{div}(f_i)$ with $a_i\in\mathbb R$ and $f_i\in\mathbb C(X)^*$. Restricting to the finitely many prime divisors appearing in $D$ and in the divisors of the $f_i$ gives a finite linear system with rational coefficients and rational right-hand side. Since it has the real solution $(a_i)$, it has a rational solution. Hence $D$ is a $\mathbb Q$-linear combination of principal divisors.
\end{proof}

\begin{lem}\label{lem:cyclic}
Let $X$ be a normal projective variety and let $\cF$ be a foliation with $K_\cF\sim_{\mathbb Q}0$. Let $r$ be minimal with $rK_\cF\sim0$. Let $\pi:Y\to X$ be the index one cover of $\cF$, and set $\cG:=\pi^{-1}\cF$. Then $\cO_Y(K_\cG)\cong\cO_Y$, and its tautological generator is an eigenvector for $\Gal(Y/X)\cong\mu_r$ with primitive character of order $r$.
\end{lem}
\begin{proof}
Let $D:=K_\cF$. The cyclic cover is defined by
\[
    \pi_*\cO_Y=\bigoplus_{i=0}^{r-1}\cO_X(-iD),
\]
with multiplication induced by a choice of trivialization $\cO_X(rD)\cong\cO_X$. The degree one summand gives a tautological section of $\pi^{[*]}\cO_X(D)$ whose $r$-th power is $1$. Since the cover is quasi-\'etale, one has
\[
    \cO_Y(K_\cG)\cong\pi^{[*]}\cO_X(K_\cF).
\]
The Galois group acts through the grading. Hence the tautological generator has character equal to the standard character of $\mu_r$ or its inverse, and therefore has exact order $r$.
\end{proof}

\begin{thm}[{\cite[Proposition~8.14]{Dru21}}]\label{thm:druel-product}
For a normal projective klt variety \(V\) and a foliation \(\cG\) on \(V\) with canonical singularities and \(K_\cG\) \(\mathbb Q\)-Cartier with \(K_\cG\equiv 0\), there exist normal projective klt varieties \(Y_0,Z_0\) and a finite quasi-\'etale cover
\(u\colon Y_0\times Z_0\to V\) such that \(u^{-1}\cG\) is the pullback by the projection \(Y_0\times Z_0\to Y_0\) of a foliation \(\cH_0\) on \(Y_0\), no positive-dimensional algebraic subvariety tangent to \(\cH_0\) passes through a general point of \(Y_0\), \(K_{Z_0}\sim 0\), \(\cH_0\) has canonical singularities, and \(K_{\cH_0}\equiv 0\).
\end{thm}

\begin{thm}[{\cite[Theorem~1.1]{DO22}}]\label{thm:DO22-structure}
For a normal projective klt variety \(W\) and a codimension one foliation \(\cJ\) on \(W\) with canonical singularities and \(K_\cJ\equiv 0\), one of the following holds.
\begin{enumerate}
\item There exist a smooth projective curve \(C\), a normal projective klt variety \(W'\) with \(K_{W'}\sim 0\), and a finite quasi-\'etale cover \(W'\times C\to W\) such that the inverse-image foliation is induced by the projection \(W'\times C\to C\).
\item There exist normal projective klt varieties \(W_1,W_2\), a finite quasi-\'etale cover \(W_1\times W_2\to W\), and a foliation \(\cH_2\) on \(W_2\) isomorphic to \(\cO_{W_2}^{\dim W_2-1}\) such that the inverse-image foliation is the pullback of \(\cH_2\) by the projection \(W_1\times W_2\to W_2\); moreover \(K_{W_1}\sim 0\), \(W_2\) is an equivariant compactification of a connected commutative algebraic group, and \(\cH_2\) is induced by a codimension one Lie subgroup.
\end{enumerate}
\end{thm}

\section{Reduction}

The goal of this section is to prove the following proposition which effectively reduces the question to the case when $X$ is $\mathbb Q$-factorial klt of dimension $3$, $\Ff$ is canonical and not algebraically integrable, and $B=0$. Classification results in \cite{Dru21,DO22} can be applied after that.

\begin{prop}\label{prop:reduction}
Assume that Theorem \ref{thm:main-pair} holds when all the following conditions hold:
\begin{center}
$X$ is $\mathbb Q$-factorial klt, $\dim X=3$, $\Ff$ is canonical, $B=0$, and $2\geq\rk\Ff>\rk_{\alg}\Ff\geq 0$.
\end{center}
Then Theorem \ref{thm:main-pair} holds.
\end{prop}
\begin{proof}
We divide the proof in several steps.

\medskip

\noindent\textbf{Step 1.} We reduce to the case when $(X,\Ff,B)$ is $\mathbb Q$-factorial dlt in the sense of \cite[Definition 5.2]{LLM23}. In particular, in the rest of the proof, we may assume that $X$ is $\mathbb Q$-factorial klt.

We may assume that $\rk\Ff>0$, otherwise the theorem is trivial. We may assume that $\rk\Ff<\dim X$, otherwise the theorem follows from \cite[Theorem 1.8.7]{Xu20}. By \cite[Theorem 1.1]{LLM23}, we may assume that $\Ii$ is a finite set. Possibly replacing $\Ii$ with $\Ii\cup\{1\}$, we may assume that $1\in\Ii$. Possibly replacing $(X,\Ff,B)$ with a $\mathbb Q$-factorial dlt modification in the sense of \cite[Theorem 5.8]{LLM23}, we may assume that $(X,\Ff,B)$ is $\mathbb Q$-factorial dlt in the sense of \cite[Definition 5.2]{LLM23}.

\medskip

\noindent\textbf{Step 2.} We reduce to the case when $\Ff$ is not algebraically integrable and $\dim X=3$. If $\Ff$ is algebraically integrable, then we let $g: (Y,\Ff_Y,B_Y;G)/Z\rightarrow (X,\Ff,B)$ be a $\mathbb Q$-factorial ACSS modification associated with contraction $f: Y\rightarrow Z$ \cite[Definition 3.3]{LMX25}. We have
\begin{equation}
    K_{\Ff_Y}+B_Y=g^*(K_{\Ff}+B)\equiv 0.
\end{equation}
By \cite[Proposition 8.1]{LLM23} there exists a positive integer $I$ depending only on $\Ii$ such that $I(K_{\Ff_Y}+B_Y)\sim 0$, hence $I(K_{\Ff}+B)\sim 0$ and the proposition follows. In the following, we may assume that $\Ff$ is not algebraically integrable. 

By \cite[Lemma 7.2]{LLM23} (based on \cite[Theorem 1]{Per05}), we may assume that $\dim X=3$. 

\medskip

\noindent\textbf{Step 3.} We reduce to the case when $\dim X=3,\rk\Ff=2$, and $\rk_{\alg}\Ff=1$.

Assume that $\Ff$ is purely transcendental. Then we may assume that $B\not=0$ or $\Ff$ is not canonical. If $B\not=0$, then $K_{\Ff}$ is not pseudo-effective, which contradicts \cite[Theorem 1.1]{CP19}. Thus $\Ff$ is not canonical. Let $h: X'\rightarrow X$ be a resolution of $\Ff$ which exists as $\dim X\leq 3$ (see \cite{Sei68,Can04,MP13}, or \cite[Theorem 4.5]{LLM23} for a summary). Then $K_{\Ff'}$ is not pseudo-effective, so $\Ff'$ has non-trivial algebraic part by \cite{CP19}, hence $\Ff$ has non-trivial algebraic part, which is also a contradiction. Thus we may assume that $\Ff$ is not purely transcendental. Since $\dim X\leq 3$ and $\Ff$ is neither purely transcendental nor algebraically integrable, we have $\dim X=3,\rk\Ff=2$, and $\rk_{\alg}\Ff=1$. 

\medskip

\noindent\textbf{Step 4.} We reduce to the case when $B=0$.

If $B\not=0$, then by \cite[Lemma 8.2]{LLM23}, we may assume that there exists a contraction $f: X\rightarrow Z$ such that $\dim X=2$ and $\Ff=f^{-1}\Ff_Z$ for some foliation $\Ff_Z$ on $Z$, and the theorem follows from \cite[Proposition 8.1]{LLM23}. Thus we may assume that $B=0$. 

\medskip

\noindent\textbf{Step 5.} We conclude the proof in this step. By our assumption, we may assume that $\Ff$ is not canonical. We let $E$ be a prime divisor over $X$ such that $a(\Ff,E)<0$ and let $W:=\Center_XE$. Since $\Ff$ is lc, $E$ is not $\Ff$-invariant. Since $B=0$, $\dim W\leq 1$. Since $(X,\Ff,B)$ is $\mathbb Q$-factorial dlt, $\Ff$ is non-dicritical \cite[Theorem 11.3]{CS21}, so $\dim W=1$ by the definition of non-dicritical \cite[Definition 2.10]{CS21}. If $W$ is tangent to $\Ff$ (cf. \cite[Definition 2.24]{Spi20}), then there exists a prime divisor $E'$ over $X$ such that $\Center_XE'=W$ and $E'$ is $\Ff$-invariant. Thus $E$ is $\Ff$-invariant by \cite[Remark 2.16]{CS21}, a contradiction. Hence $W$ is transverse to $\Ff$. 

Since $(X,\Ff)$ is $\mathbb Q$-factorial dlt, $X$ is klt \cite[Theorem 11.3]{CS21}. We let $F$ be a prime divisor over $X$ such that $\Center_XF=W$ and
\begin{equation}
 1>a(F,X)+1=\mld_W(X):=1+\inf\{a(D,X)\mid \Center_XD=W\},
\end{equation}
where $D$ runs through all prime divisors over $X$. Since $\dim W=1$, by the ACC for mlds on surfaces \cite[Theorem 3.2]{Ale93}, we have that $a(F,X)$ belongs to an ACC set in $[-1,0]\cap\mathbb Q$ depending only on $\Ii$. By \cite[Lemma 3.11]{Spi20}, we have
$$a(F,X)=a(F,\Ff)\leq a(E,X)<0.$$
We let $h: X'\rightarrow X$ be the extraction of $F$ whose existence is guaranteed by \cite[Corollary 1.4.3]{BCHM10}. Then we have
\begin{equation}
K_{\Ff'}+bF=h^*K_{\Ff}\equiv 0,
\end{equation}
where $\Ff'=h^{-1}\Ff$ and $b:=-a(F,X)$ belongs to a DCC set depending only on $\Ii$. By \textbf{Step 4}, there exists a positive integer $I$ depending only on $\Ii$ such that $I(K_{\Ff'}+bF)\sim 0$, hence $IK_{\Ff}\sim 0$. The proposition follows.
\end{proof}

\section{Preliminaries on Lie theory}

The proof in the non-algebraically integrable cases uses elementary structure theory of commutative algebraic groups and their equivariant compactifications. We collect the required facts in this section.

\medskip

\noindent\textbf{Semi-abelian varieties and periods.} A \emph{semi-abelian variety} is a connected commutative algebraic group $S$ fitting into an exact sequence
\[
    0\longrightarrow G\longrightarrow S\longrightarrow A\longrightarrow 0,
\]
where $G\cong \mathbb G_m^t$ is an algebraic torus and $A$ is an abelian variety. The analytic exponential map of $S$ is denoted by
\[
    \exp_S:\Lie S\longrightarrow S(\C).
\]
Its kernel
\[
    \Lambda_S:=\ker(\exp_S)
\]
is called the \emph{period group} of $S$. If $\dim A=g$, then $\Lambda_S$ is a free abelian group of rank $t+2g$, and it spans $\Lie S$ over $\C$.

\begin{lem}\label{lem:lattice}
Let $S$ be a connected semi-abelian variety over $\C$ with $\dim S\le3$, and let
\[
    \Lambda=\ker(\exp_S:\Lie S\to S(\C)).
\]
Let $T:\Lie S\to\Lie S$ be a finite-order $\mathbb C$-linear automorphism such that $T(\Lambda)=\Lambda$. Let $W\subset\Lie S$ be a $T$-stable linear subspace. Then $T|_W$ induces an automorphism of the one-dimensional vector space $\det(W^\vee)$; if this automorphism is multiplication by $\chi_W(T)\in\C^*$, then $\chi_W(T)$ has order at most $30$.
\end{lem}

\begin{proof}
Write $0\to G\to S\to A\to0$, where $G\cong\mathbb G_m^t$ and $A$ is an abelian variety of dimension $g$. Since $t+g=\dim S\le3$, the period group has rank
\[
    \rho=t+2g\le6.
\]
It spans $\Lie S$ over $\C$. Thus $T$ is determined by its restriction to $\Lambda$, and the order of $T$ equals the order of the finite-order element $T|_\Lambda\in\operatorname{GL}_\rho(\Z)$.

The characteristic polynomial of $T|_\Lambda$ is a product of cyclotomic polynomials $\Phi_d$, with
\[
    \sum \varphi(d)\le \rho\le 6.
\]
The possible $d$ with $\varphi(d)\le6$ are
\[
    1,2,3,4,5,6,7,8,9,10,12,14,18.
\]
A direct check under the constraint $\sum\varphi(d)\le6$ shows that the largest possible least common multiple of the occurring $d$ is $30$. Hence $\operatorname{ord}(T)\le30$, where $\operatorname{ord}(T)$ is the smallest positive integer $n$ such that $T^n=\operatorname{id}$. Since the action on $\det(W^\vee)$ is induced by $T|_W$, its order divides $\operatorname{ord}(T)$.
\end{proof}

\begin{defn}[Translation-invariant foliation]
Let $A$ be a connected algebraic group, and let $W\subset \Lie A$ be a linear subspace. We write $TA$ for the tangent bundle of the smooth variety $A$. Left-invariant vector fields give the canonical trivialization
\[
    TA \cong \Lie A\otimes_{\mathbb C}\cO_A.
\]
Under this trivialization, $W$ defines a locally free subsheaf
\[
    \cG_W:=W\otimes_{\mathbb C}\cO_A\subset TA .
\]
If $W$ is closed under the Lie bracket, then $\cG_W$ is a foliation on $A$, called the \emph{translation-invariant foliation} induced by $W$. In particular, if $A$ is commutative, then every linear subspace $W\subset\Lie A$ induces such a foliation.
\end{defn}

\begin{defn}
Let $A$ be a connected algebraic group and let $W\subset\Lie A$ be a linear subspace. We say that $W$ \emph{algebraically generates} $A$ if no proper connected algebraic subgroup $B\subsetneq A$ satisfies $W\subset\Lie B$.
\end{defn}

\begin{lem}\label{lem:rank-zero-generation}
Let $r$ be a positive integer. Let $X$ be a normal projective variety and let $A$ be a connected commutative algebraic group acting on $X$ with dense open orbit $U\cong A$. Let $W\subset\Lie A$ be an $r$-dimensional linear subspace, and let $\Ff$ be the foliation induced on $U$ by $W$.

If $\dim A=r+1$ and $\Ff$ is purely transcendental, then $W$ algebraically generates $A$.
\end{lem}
\begin{proof}
If $B\subsetneq A$ is connected and $W\subset\Lie B$, then
\[
    r=\dim W\le \dim B<\dim A=r+1,
\]
hence $\dim B=r$ and $\Lie B=W$. For a general point $x\in U$, the orbit $B\cdot x$ is an $r$-dimensional algebraic subvariety tangent to $\Ff$. This contradicts the purely transcendental assumption.
\end{proof}

\begin{lem}\label{lem:generation-after-cover}
Let $\phi:A'\to A$ be a finite surjective homomorphism between connected commutative algebraic groups, and let $W\subset\Lie A$ be a linear subspace. Set
\[
    W':=(d\phi)^{-1}(W)\subset\Lie A'.
\]
If $W$ algebraically generates $A$, then $W'$ algebraically generates $A'$.
\end{lem}
\begin{proof}
Assume that $B'\subset A'$ is connected and $W'\subset\Lie B'$. Then $B:=\phi(B')$ is connected and $W\subset\Lie B$. Thus $B=A$. Since $\phi$ is finite, $\dim B'=\dim B=\dim A=\dim A'$, so $B'=A'$.
\end{proof}

\begin{lem}[Finite refinements of equivariant compactifications]\label{lem:finite-refinement-equivariant}
Let $P$ be a normal projective variety with an action of a connected commutative algebraic group $A$, and suppose that $P$ contains a dense open orbit $U\cong A$. Let $h:R\to P$ be a finite surjective quasi-\'etale morphism, with $R$ normal and irreducible, and set $U_R:=h^{-1}(U)$. After choosing a point of $U_R$ over the identity of $A$, the open set $U_R$ has a unique structure of connected commutative algebraic group $A_R$ such that
\[
    \phi_R:=h|_{U_R}:A_R\to A
\]
is a finite \'etale homomorphism. Moreover, the translation action of $A_R$ on $U_R=A_R$ extends uniquely to an algebraic action $A_R\times R\to R$, and $h$ is equivariant with respect to $\phi_R$.
\end{lem}

\begin{proof}
Identify $U$ with $A$ so that the chosen point of $U$ is the identity. The restriction
\[
    h_U:=h|_{U_R}:U_R\to U\cong A
\]
is finite. Since $h$ is quasi-\'etale, $h_U$ is \'etale at every codimension one point of $U_R$; because the target $A$ is smooth, Zariski--Nagata purity \cite[Expos\'e~X, Th\'eor\`eme~3.1]{SGA1} implies that $h_U$ is finite \'etale. As $R$ is irreducible, $U_R$ is connected. Choose $e_R\in U_R$ over the identity $e_A\in A$.

The pointed finite \'etale cover $h_U:(U_R,e_R)\to(A,e_A)$ corresponds to a finite-index subgroup
\[
    H=(h_U)_*\pi_1(U_R,e_R)\subset\pi_1(A,e_A).
\]
Since $A(\C)$ is a connected commutative topological group, multiplication on $A$ induces addition on $\pi_1(A,e_A)$ and inversion induces multiplication by $-1$. Hence $H$ is stable under the operations needed to lift multiplication and inversion, i.e. $H+H=H$ and $H^{-1}=H$. By the lifting criterion for finite \'etale covers, $m_A\circ(h_U\times h_U)$ and $i_A\circ h_U$ lift uniquely to pointed morphisms
\[
    m_R:U_R\times U_R\to U_R,
    \qquad
    i_R:U_R\to U_R.
\]
The group axioms, and commutativity, follow because both sides of each axiom are pointed lifts of the same morphism to $A$. This gives the unique connected commutative algebraic group structure on $U_R$ for which $h_U$ is a homomorphism; denote this group by $A_R$ and write $\phi_R=h_U$.

It remains to extend translations to $R$. Put $B:=A_R\times R$, which is normal, and define
\[
    F:B\to P,
    \qquad
    F(a,x):=\phi_R(a)\cdot h(x).
\]
Let $T:=B\times_{P,h}R$. The projection $T\to B$ is finite. On the dense open subset $B^\circ:=A_R\times A_R\subset B$, the group law of $A_R$ defines a section
\[
    s^\circ:B^\circ\to T,
    \qquad
    (a,b)\longmapsto(a,b,ab),
\]
because $h(ab)=\phi_R(a)\phi_R(b)=\phi_R(a)\cdot h(b)$. Let $T_0$ be the scheme-theoretic closure of $s^\circ(B^\circ)$ in $T$, and let $T_0^\nu\to T_0$ be the normalization. Then $T_0^\nu\to B$ is finite and birational; since $B$ is normal, it is an isomorphism. Composing $B\cong T_0^\nu$ with the projection to $R$ gives a regular morphism
\[
    \mu_R:A_R\times R\to R.
\]
It restricts to the group law on $A_R\times A_R$ and satisfies
\[
    h(\mu_R(a,x))=\phi_R(a)\cdot h(x).
\]
The identity and associativity axioms hold on the dense open subset $A_R\times A_R\times A_R$ and hence everywhere, since $R$ is separated. Thus $\mu_R$ is an algebraic action extending translations. The displayed identity gives equivariance of $h$, and uniqueness follows because two such actions agree on the dense open subset $A_R\times A_R$.
\end{proof}

\noindent\textbf{Unipotent directions.}
We use the following standard structure facts for algebraic groups over $\C$.
An algebraic group $G$ is called \emph{unipotent} if every nonzero finite-dimensional representation of $G$ has a nonzero fixed vector; equivalently, every finite-dimensional representation of $G$ is unipotent in the sense that, after choosing a basis, its image is contained in an upper triangular unipotent group \cite[Proposition~14.3]{MilneAG}.
Let $A$ be a connected commutative algebraic group over $\C$. By the Barsotti--Chevalley theorem, there is a unique connected affine normal subgroup
\[
    A_{\mathrm{aff}}\subset A
\]
such that $B:=A/A_{\mathrm{aff}}$ is an abelian variety \cite[Theorem~8.27]{MilneAG}. Since $A$ is commutative, $A_{\mathrm{aff}}$ is a smooth connected affine solvable group. By the structure theorem for smooth connected solvable groups over a perfect field, $A_{\mathrm{aff}}$ has a normal unipotent subgroup $(A_{\mathrm{aff}})^u$ which contains every unipotent algebraic subgroup of $A_{\mathrm{aff}}$, and
\[
    A_{\mathrm{aff}}/(A_{\mathrm{aff}})^u
\]
is a torus \cite[Theorem~16.33]{MilneAG}. Every unipotent subgroup of $A$ is affine, hence is contained in $A_{\mathrm{aff}}$. Thus
\[
    A^{\mathrm{unip}}:=(A_{\mathrm{aff}})^u
\]
is the maximal connected unipotent subgroup of $A$. It is characteristic in $A$, because $A_{\mathrm{aff}}$ is uniquely characterized by the Barsotti--Chevalley theorem and $(A_{\mathrm{aff}})^u$ is uniquely characterized inside $A_{\mathrm{aff}}$ as the maximal connected unipotent subgroup. Moreover, since $A$ is commutative, $A^{\mathrm{unip}}$ is a vector group, i.e.
\[
    A^{\mathrm{unip}}\simeq \mathbb G_a^n
\]
for some $n$ \cite[Corollary~14.33]{MilneAG}. Finally, quotienting the Barsotti--Chevalley sequence by $A^{\mathrm{unip}}$ gives an exact sequence
\[
    0\to A_{\mathrm{aff}}/A^{\mathrm{unip}}
    \to A/A^{\mathrm{unip}}
    \to B\to 0,
\]
where $A_{\mathrm{aff}}/A^{\mathrm{unip}}$ is a torus and $B$ is an abelian variety. Therefore $A/A^{\mathrm{unip}}$ is semi-abelian, i.e. an extension of an abelian variety by a torus.

\begin{lem}[No unipotent direction]\label{lem:no-unipotent-direction}
Let $A$ be a connected commutative algebraic group acting on a normal projective variety $R$ with dense open orbit $U\cong A$. Let $\cG$ be induced on $U$ by a linear subspace $W\subset\Lie A$. If $\cG$ has algebraic rank zero, then
\[
    W\cap\Lie A^{\mathrm{unip}}=0.
\]
\end{lem}

\begin{proof}
Suppose that $0\ne u\in W\cap\Lie A^{\mathrm{unip}}$. Since $A^{\mathrm{unip}}\simeq\mathbb G_a^N$, the subspace $\C u\subset\Lie A^{\mathrm{unip}}$ gives a one-dimensional subgroup $U_1\simeq\mathbb G_a$ of $A^{\mathrm{unip}}$, under the exponential isomorphism (see \cite[Proposition 14.32]{MilneAG}). For a general point $p\in U$, the orbit $U_1\cdot p$ is a positive-dimensional algebraic subvariety tangent to $\cG$, because its tangent direction is generated by $u\in W$. This contradicts algebraic rank zero.
\end{proof}

\begin{lem}[No unipotent direction after a finite group cover]
\label{lem:no-unipotent-after-cover}
Let $\phi:A'\to A$ be a finite surjective homomorphism of connected commutative algebraic groups over $\C$. Let $W\subset\Lie A$ be a linear subspace and set
\[
    W':=(d\phi)^{-1}(W)\subset\Lie A'.
\]
If $W\cap\Lie A^{\mathrm{unip}}=0$, then
\[
    W'\cap\Lie (A')^{\mathrm{unip}}=0.
\]
\end{lem}

\begin{proof}
Let $v\in W'\cap\Lie(A')^{\mathrm{unip}}$. Then $d\phi(v)\in W$. By the discussion above, the image $\phi((A')^{\mathrm{unip}})$ is unipotent: indeed, if $\iota:\phi((A')^{\mathrm{unip}})\hookrightarrow \operatorname{GL}(V)$ is a faithful representation, then $\iota\circ\phi|(A')^{\mathrm{unip}}$ is a representation of the unipotent group $(A')^{\mathrm{unip}}$, so every element of $\phi((A')^{\mathrm{unip}})$ is represented by a unipotent matrix. Therefore $\phi((A')^{\mathrm{unip}})$ is a connected unipotent subgroup of $A$, hence is contained in $A^{\mathrm{unip}}$. Consequently,
\[
    d\phi(v)\in W\cap\Lie A^{\mathrm{unip}}=0.
\]
Since $\phi$ is finite over $\C$, $\ker(d\phi)=\Lie(\ker\phi)=0$. Hence $v=0$.
\end{proof}

\begin{lem}[Normalizer of the translation group]
\label{lem:normalizer-translation}
Let $A$ be a connected algebraic group, and let $f:A\to A$ be an automorphism of the variety $A$. Assume that $f$ normalizes the translation action of $A$, i.e. for every $a\in A$ there exists $\rho(a)\in A$ such that
\[
    f\circ t_a\circ f^{-1}=t_{\rho(a)},
\]
where $t_a$ denotes left translation by $a$. Then $\rho:A\to A$ is an algebraic group automorphism and
\[
    f(x)=\rho(x)\cdot f(e)
\]
for all $x\in A$. In particular, in additive notation,
\[
    f(x)=\rho(x)+f(0).
\]
If $f$ has finite order, then $\rho$ has finite order.
\end{lem}

\begin{proof}
The element $\rho(a)$ is uniquely determined. Conjugation preserves composition of translations, hence
\[
    t_{\rho(ab)}=f\circ t_{ab}\circ f^{-1}=f\circ t_a\circ t_b\circ f^{-1}=t_{\rho(a)}\circ t_{\rho(b)}=t_{\rho(a)\rho(b)}.
\]
Thus $\rho(ab)=\rho(a)\rho(b)$.

For the identity element $e\in A$,
\[
    f(a)=f(t_a(e))=(f\circ t_a\circ f^{-1})(f(e))=\rho(a)\cdot f(e),
\]
so $\rho(a)=f(a)\cdot f(e)^{-1}$. This formula shows that $\rho$ is a morphism; applying the same argument to $f^{-1}$ shows that it is an algebraic group automorphism.

If $f^n=\operatorname{id}$, then applying the conjugation relation $n$ times gives $t_a=t_{\rho^n(a)}$ for every $a\in A$. Hence $\rho^n=\operatorname{id}$.
\end{proof}

\begin{lem}[Intrinsic normalizer and determinant character]\label{lem:rank-two-intrinsic-normalizer}
Let $R$ be a normal projective variety, and let $A$ be a connected commutative algebraic group acting on $R$ with dense open orbit $U\cong A$. Let $\cG\subset TR$ be a saturated rank $k$ foliation. Suppose that a $k$-dimensional subspace $W\subset\Lie A$ induces an isomorphism
\[
    W\otimes\cO_R\cong\cG
\]
via the infinitesimal $A$-action, and suppose that $W$ algebraically generates $A$. Let $\Gamma\subset\Aut(R)$ be a finite group preserving $\cG$. Then every $\gamma\in\Gamma$ normalizes $A$, preserves $U$, and acts on $U\cong A$ by an affine automorphism
\[
    x\longmapsto L_\gamma(x)+a_\gamma,
\]
where $L_\gamma\in\Aut_{\mathrm{gp}}(A)$ has finite order. Moreover, $dL_\gamma$ preserves $W$, and under the natural identification
\[
    H^0(R,\cO_R(K_\cG))\cong\det(W^\vee),
\]
the action of $\gamma$ is the determinant-dual action induced by $dL_\gamma|_W$.
\end{lem}

\begin{proof}
By the Matsumura--Oort identification \cite[Lemma~3.4]{MO67},
\[
    H^0(R,TR)=\Lie\Aut^0(R).
\]
Since $W\otimes\cO_R\cong\cG$ and $H^0(R,\cO_R)=\C$, we have $H^0(R,\cG)=W$. Let $A^{\mathrm{int}}\subset\Aut^0(R)$ be the identity component of the Zariski closure of the connected complex Lie subgroup generated by the flows of $H^0(R,\cG)$. These vector fields are the infinitesimal translations coming from the $A$-action, so $A^{\mathrm{int}}\subset A$. Its Lie algebra contains $W$; since $W$ algebraically generates $A$, it follows that $A^{\mathrm{int}}=A$.

For $\gamma\in\Gamma$, the pushforward $\gamma_*$ preserves $H^0(R,\cG)$ because $\gamma$ preserves $\cG$. Hence $\gamma$ normalizes the flow closure $A$. The dense open $A$-orbit is unique, so $\gamma(U)=U$. After identifying $U$ with $A$, the restriction $\gamma|_U$ normalizes translations; by Lemma~\ref{lem:normalizer-translation}, it has the form
\[
    \gamma|_U(x)=L_\gamma(x)+a_\gamma
\]
with $L_\gamma\in\Aut_{\mathrm{gp}}(A)$. Since $\Gamma$ is finite, $L_\gamma$ has finite order. The action on translation-invariant vector fields is $dL_\gamma$, and therefore $dL_\gamma(W)=W$.

Dualizing $W\otimes\cO_R\cong\cG$ and taking determinants gives
\[
    \cO_R(K_\cG)\cong\det(W^\vee)\otimes\cO_R.
\]
Taking global sections gives the stated one-dimensional representation, and the induced action is the determinant-dual action of $dL_\gamma|_W$.
\end{proof}
\section{The rank one case}\label{sec:rank-one}

In this section we prove the following bounded index statement for rank one foliations.

\begin{prop}\label{prop:rank-one}
Let $X$ be a $\mathbb Q$-factorial klt projective threefold and let $\Ff$ be a canonical foliation on $X$ such that $\rk\Ff=1$, $\rk_{\alg}\Ff=0$, and $K_{\Ff}\equiv 0$. Then $rK_{\Ff}\sim 0$ for some integer $1\leq r\leq 30$.
\end{prop}
\begin{proof}
Let \(r\) be the smallest positive integer such that $rK_{\mathcal F}\sim 0.$ We only need to show that $r\leq 30$. 

Let $\pi: Y\rightarrow X$ be the index $1$ cover associated to $K_{\Ff}$. Then $\pi$ is finite quasi-\'etale and the Galois group $\Gal(Y/X)$ is naturally identified with $\mu_r$. Let $\delta$ be a generator of $\Gal(Y/X)$ and let $\Gg=\pi^{-1}\Ff$, then we have 
\begin{equation}
    0\sim K_{\Gg}=\pi^*K_{\Ff}
\end{equation}
and the degree $1$ summand $\mathcal{O}_X(-K_{\Ff})\subset\pi_*\mathcal{O}_Y$ gives a canonical trivializing section
$\eta\in H^0(Y,K_{\Gg})$. Then $\eta$ is an eigenvector satisfying that
$$\delta^*\eta=\lambda\eta$$
where $\lambda=e^{2\pi i/r}$. In particular, $\lambda$ has order $r$. Since $\rk\Gg=1$, by the definition of $K_{\Gg}$, we have
\begin{equation}
    \mathcal{O}_Y(K_{\Gg})=\Gg^{\vee}
\end{equation}
hence the dual of $\eta$ induces a generator
$$v\in H^0(Y,\Gg)\subset H^0(Y,TY)$$
and we have
$$\delta_*v=\lambda^{-1}\cdot v.$$
We now consider $v$ as a global vector field on $Y$. Since \(\Aut^0(Y)\) is the neutral component of \(\Aut_Y\), the Lie algebra of $\Aut^0(Y)$ and $\Aut_Y$ are the same. By \cite[Lemma 3.4]{MO67}, 
$$H^0(Y,TY)=\Lie\Aut^0(Y).$$
Let 
$$W:=\mathbb Cv \quad \text{and}\quad B:=\exp(\mathbb Cv)\subset\Aut^0(Y)(\mathbb C),$$
and let $A:=\overline{B}\subset\Aut^0(Y)$ the Zariski closure of $B$ in $\Aut^0(Y)$. Then $B$ is connected and commutative, so $A$ is connected and commutative. We have $W\subset\Lie A$.

Let $U\subset A$ be the maximal connected unipotent subgroup. Since $A$ is commutative, $U$ is characteristic in $A$, so $U$ is $\delta$-invariant. Let $S:=A/U.$
Then $S$ is a semi-abelian variety. We let 
$$p:\quad \Lie A\rightarrow\Lie S$$
be the quotient map. By Lemma \ref{lem:no-unipotent-direction}, $W\not\subset U$, hence $p(W)\not=0$.

Next we define a subgroup $K\subset S$. Let $y\in Y$ be a general closed point and let 
$$H_y:=\Stab_A(y)^0$$
be the identity component of the stabilizer of $y$. Let $K_y\subset S$ be the connected image of $H_yU$ in $S=A/U$. Since connected algebraic
subgroups of a semi-abelian variety form a countable set, and the loci
where \(K_y\) is equal to a fixed subgroup are constructible, there exists
a dense open subset $Y_{\mathrm{gen}}\subset Y$ and a connected algebraic subgroup $K\subset S$ such that $K_y=K$ for every \(y\in Y_{\mathrm{gen}}\). We have that $K$ is $\delta$-invariant.

Since $U$ is $\delta$-invariant, there exists an automorphism $\bar\delta: S\rightarrow S$ induced by $\delta$. For any $y\in Y$, we have
$$H_{\delta(y)}=\delta H_y\delta^{-1}$$
hence
$$K_{\delta(y)}=\overline\delta(K_y).$$
Since we work over $\C$, there exists $y\in Y_{\mathrm{gen}}\cap\delta^{-1}(Y_{\mathrm{gen}})$, and we have
$$K=\overline\delta(K).$$
We define
\[
\overline S:=S/K
\]
and let
\[
q:\Lie A\to \Lie\overline S
\]
be the composition
\[
\Lie A\xrightarrow{p}\Lie S\to \Lie(S/K).
\]
of the quotient map and $p$. 
\begin{claim}\label{claim:qwnot0}
    $q(W)\not=0$.
\end{claim}
\begin{proof}
Suppose that $q(W)=0$. Then we may pick $y\in Y_{\mathrm{gen}}$ such that $Y$ is smooth near $y$, $\Gg$ is regular near $y$, and $v(y)\not=0$. Let $0\not=w\in W$ be the element corresponding to $v$. Then $p(w)\in\Lie K$. Since $K$ is the connected image of $H_yU$ in $S=A/U$, there exists $z\in\Lie H_y$ such that
$p(z)=p(w)$. Let $u:=w-z$, then $p(u)=0$, hence $u\in\Lie U$. 
Since $v(y)\not=0$, the infinitesimal action of $w$ at $y$ does not vanish. Since $z\in\Lie H_y$, the infinitesimal action of $z$ at $y$ vanishes. Thus $u=w-z\not=0$, so the line $\mathbb Cu\subset\Lie U$ integrates to a $1$-dimensional algebraic subgroup
$$U_y\cong\mathbb G_a$$
on $U$. Let $C_y:=\overline{U_y\cdot y}\subset Y$. Then $C_y$ is an algebraic curve. Since $A$ is commutative, $H_y$ commutes with $U_y$. Since $z\in\Lie H_y$, the vector field induced by $z$ vanishes along $U_y\cdot y$. Thus the vector field induced by $w$ and $u$ are equal. The vector field induced by $w$ is the vector field induced by $v$. Thus $C_y$ is tangent to $\Gg$. Since $y$ is general, $\Gg$ is algebraically integrable, so $\Ff$ is algebraically integrable. A contradiction.
\end{proof}

\noindent\textit{Proof of Proposition \ref{prop:rank-one} continued.} Let $L:\overline S\to\overline S$ be the automorphism induced by $\bar\delta$. By Claim \ref{claim:qwnot0}, $0\not=q(W)\subset\Lie\overline S$. Then $q(W)$ is $L$-stable and the scalar of the action of $L$ is $\lambda^{-1}$. 

Pick a general closed point $y\in Y_{\mathrm{gen}}$. We have
\begin{equation}
   \dim\overline{S}=\dim (S/K)=\dim (A/H_yU)\leq\dim (A/H_y)=\dim (A\cdot y)\leq\dim Y=3. 
\end{equation}
By Lemma~\ref{lem:lattice}, the order of the $\lambda^{-1}$ is $\leq 30$. Thus $r\leq 30$ and the proposition follows.
\end{proof}

\section{Rank two: not purely transcendental case}\label{sec:case-d}

Throughout this section, $X$ is a $\mathbb Q$-factorial klt projective threefold over $\mathbb C$, and $\cF\subset TX$ is a rank two foliation with canonical singularities and $K_\cF\equiv 0$. We treat the case where the algebraic part of $\cF$ has rank one.

\begin{prop}\label{prop:rk2-alg1}
Let $X$ be a $\mathbb Q$-factorial klt projective threefold and let $\cF$ be a rank two foliation on $X$ with canonical singularities. Assume that $K_\cF\equiv 0$ and that the algebraic part of $\cF$ has rank one. Then there exists an integer $r$ with $1\le r\le 30$ such that
\[
    rK_{\cF}\sim 0.
\]
Equivalently, if $r$ is the exact torsion index of $K_\cF$, then $r\le 30$.
\end{prop}

\subsection{Generated product model}

\begin{lem}[Source extraction]\label{lem:case-d-source}
Under the hypotheses of Proposition~\ref{prop:rk2-alg1}, there exist:
\begin{enumerate}
    \item a connected normal projective klt surface $S$;
    \item a smooth connected projective genus one curve $E$;
    \item a finite surjective quasi-\'etale morphism
    \[
        q_0:S\times E\to X;
    \]
    \item a saturated rank one foliation $\cH$ on $S$ such that $q_0^{-1}\cF$ is the pullback of $\cH$ by the projection $S\times E\to S$;
    \item a connected normal projective surface $Z$, a finite surjective quasi-\'etale morphism $g:Z\to S$, a saturated rank one inverse-image foliation $\cH_Z$ on $Z$, a connected commutative algebraic group $A_Z$ of dimension two, a dense open orbit $U_Z\subset Z$, a point $z_0\in U_Z$, and a one-dimensional subspace $W_Z\subset \Lie A_Z$.
\end{enumerate}
such that the orbit map
\[
    A_Z\longrightarrow U_Z,
    \qquad
    a\longmapsto a\cdot z_0,
\]
is an isomorphism, and $\cH_Z|_{U_Z}$ is the translation-invariant foliation induced by $W_Z$.

Moreover, the composite
\[
    Q:=q_0\circ(g\times \operatorname{id}_E):Z\times E\to X
\]
is finite surjective quasi-\'etale, and $Q^{-1}\cF$ is the pullback of $\cH_Z$ by the projection $Z\times E\to Z$. The foliations $\cH$ and $\cH_Z$ are canonical, satisfy $K_\cH\equiv0$ and $K_{\cH_Z}\equiv0$, and have algebraic rank zero.
\end{lem}

\begin{proof}
Apply Theorem~\ref{thm:druel-product} to $(X,\cF)$. We obtain normal projective klt varieties $Y_0,Z_0$, a finite quasi-\'etale cover
\[
    u:Y_0\times Z_0\to X,
\]
and a foliation $\cH_0$ on $Y_0$ such that $u^{-1}\cF$ is the pullback of $\cH_0$ by the projection to $Y_0$. Moreover, $\cH_0$ has algebraic rank zero, $K_{Z_0}\sim0$, $\cH_0$ is canonical, and $K_{\cH_0}\equiv0$.

Finite covers preserve algebraic rank. Since $u^{-1}\cF$ is the product foliation pulled back from $\cH_0$, its algebraic rank is
\[
    \rk_{\alg}(\cH_0)+\dim Z_0=\dim Z_0.
\]
By hypothesis this rank is one. Hence $\dim Z_0=1$ and $\dim Y_0=2$. The normal projective curve $Z_0$ satisfies $K_{Z_0}\sim0$, so it is a smooth connected projective curve of genus one. Set
\[
    S:=Y_0,
    \qquad
    E:=Z_0,
    \qquad
    q_0:=u,
    \qquad
    \cH:=\cH_0.
\]
Then $\cH$ is canonical, $K_\cH\equiv0$, and $\cH$ has algebraic rank zero.

Apply Theorem~\ref{thm:DO22-structure} to the rank one codimension one foliation $(S,\cH)$. Alternative~(1) would produce algebraic curves tangent to the pulled-back foliation through general points, contradicting algebraic rank zero. Thus alternative~(2) holds: there are normal projective klt varieties $W_1,W_2$, a finite quasi-\'etale cover $W_1\times W_2\to S$, and a foliation $\cH_2$ on $W_2$ such that the inverse-image foliation is the pullback of $\cH_2$ by the projection to $W_2$.

For a product foliation pulled back from $W_2$, the algebraic rank equals $\dim W_1+\rk_{\alg}\cH_2$. Since the pulled-back foliation has algebraic rank zero, we get $\dim W_1=0$ and $\rk_{\alg}\cH_2=0$. Since $\dim S=2$, it follows that $\dim W_2=2$.

Set $Z:=W_2$, let $g:Z\to S$ be the induced finite quasi-\'etale morphism, and set $\cH_Z:=\cH_2$. By Theorem~\ref{thm:DO22-structure}(2), $Z$ is an equivariant compactification of a connected commutative algebraic group $A_Z$ of dimension two, with dense open orbit $U_Z$, and $\cH_Z|_{U_Z}$ is induced by a one-dimensional subspace $W_Z\subset\Lie A_Z$. The numerical triviality and canonical singularities of $\cH_Z$ follow from quasi-\'etale pullback.

Finally, $g\times\operatorname{id}_E$ and hence
\[
    Q=q_0\circ(g\times\operatorname{id}_E):Z\times E\to X
\]
are finite surjective quasi-\'etale. On the codimension one locus where the finite maps are \'etale and the foliations are locally free, inverse-image foliations commute with composition, giving
\[
    Q^{-1}\cF=(p_Z)^{-1}\cH_Z,
\]
where $p_Z:Z\times E\to Z$ is the first projection. Since both sides are saturated reflexive subsheaves of $T(Z\times E)$, the equality extends globally.
\end{proof}

We now fix the notation supplied by Lemma~\ref{lem:case-d-source}. Choose an identity element on $E$, so that $E$ is an elliptic curve. Set
\[
    P:=Z\times E,
    \qquad
    \Aprod:=A_Z\times E,
    \qquad
    \Uprod:=U_Z\times E,
\]
and
\[
    \Wprod:=W_Z\oplus\Lie E\subset\Lie\Aprod.
\]
The open subset $\Uprod\subset P$ is a dense open orbit of $\Aprod$, and the pulled-back foliation $\cG_P:=Q^{-1}\cF$ is induced on $\Uprod$ by $\Wprod$.

By Lemma~\ref{lem:no-unipotent-direction}, $W_Z\cap\Lie A_Z^{\mathrm{unip}}=0$. Since an elliptic curve has no non-trivial connected unipotent subgroup,
\[
    \Wprod\cap\Lie\Aprod^{\mathrm{unip}}=0.
\]
Moreover, Lemma~\ref{lem:rank-zero-generation} shows that $W_Z$ algebraically generates $A_Z$. It follows that $\Wprod$ algebraically generates $\Aprod$: if a connected subgroup $B\subset\Aprod$ satisfies $\Wprod\subset\Lie B$, then its projections to $A_Z$ and to $E$ are both surjective, and the inclusion $\Lie E\subset\Lie B$ forces $\dim B=\dim\Aprod$.

\subsection{Proof of Proposition~\ref{prop:rk2-alg1}}

\begin{proof}[Proof of Proposition~\ref{prop:rk2-alg1}]
By the abundance theorem and Lemma~\ref{lem:R-to-Q-linear}, $K_\cF\sim_{\mathbb Q}0$. Let $r$ be the minimal positive integer such that $rK_{\Ff}\sim 0$. We shall prove that $r\le30$.

Let $\pi:Y\to X$ be the connected cyclic index one cover of $K_\cF$, and set $\cG_Y:=\pi^{-1}\cF$. By Lemma~\ref{lem:cyclic}, $\cO_Y(K_{\cG_Y})\cong\cO_Y$, and its tautological generator $\eta_Y$ has primitive deck character of order $r$.

Let $Q:P=Z\times E\to X$ be the quasi-\'etale cover constructed above. Let $R$ be the normalization of $X$ in the Galois closure of the compositum of $\mathbb C(Y)$ and $\mathbb C(P)$ over $\mathbb C(X)$. Denote the induced maps by
\[
    \rho:R\to X,
    \qquad
    p:R\to Y,
    \qquad
    h:R\to P,
\]
and set $\Gamma:=\Gal(R/X)$. Since $Y\to X$ and $P\to X$ are quasi-\'etale, the covers $R\to X$, $p$, and $h$ are finite quasi-\'etale: at codimension one points the corresponding extensions of discrete valuation rings are unramified, and this property is preserved under composita and Galois closures.

Apply Lemma~\ref{lem:finite-refinement-equivariant} to $h:R\to P$. We obtain a connected commutative algebraic group $A_R$ with dense open orbit $U_R=h^{-1}(\Uprod)$, a finite \'etale homomorphism
\[
    \phi_R:A_R\to\Aprod,
\]
and an extension of the translation action of $A_R$ to $R$ for which $h$ is equivariant. Set
\[
    W_R:=(d\phi_R)^{-1}(\Wprod)\subset\Lie A_R.
\]
Since $\phi_R$ is finite \'etale, $d\phi_R$ is an isomorphism and $W_R$ is two-dimensional. The inverse-image foliation $\cG_R:=\rho^{-1}\cF$ is induced on $U_R$ by $W_R$.

Choose a basis $w_1,w_2$ of $W_R$, and let $\widetilde w_1,\widetilde w_2$ be the corresponding global vector fields on $R$ induced by the $A_R$-action. On the codimension one locus where $h$ is \'etale and the foliations are locally free, these vector fields map to the translation-invariant vector fields on $\Aprod$ associated to $d\phi_R(w_i)\in\Wprod$, hence lie in $\cG_R$. By reflexivity,
\[
    \widetilde w_i\in H^0(R,\cG_R).
\]
The wedge
\[
    \sigma:=\widetilde w_1\wedge\widetilde w_2\in H^0(R,\det\cG_R)
\]
is nonzero at the generic point. It defines an effective Weil divisor $D$ with
\[
    \cO_R(D)\cong\det\cG_R\cong\cO_R(-K_{\cG_R}).
\]
Since $K_{\cG_R}\equiv0$, we have $D\equiv0$. An effective numerically trivial $\mathbb Q$-Cartier divisor is zero, so $D=0$. Therefore
\[
    W_R\otimes\cO_R\cong\cG_R,
    \qquad
    H^0(R,\cO_R(K_{\cG_R}))\cong\det(W_R^\vee).
\]

By Lemma~\ref{lem:no-unipotent-after-cover}, the equality $\Wprod\cap\Lie\Aprod^{\mathrm{unip}}=0$ implies
\[
    W_R\cap\Lie A_R^{\mathrm{unip}}=0.
\]
Thus $W_R$ injects into the Lie algebra of the semi-abelian quotient
\[
    S_R:=A_R/A_R^{\mathrm{unip}},
\]
which has dimension at most three. By Lemma~\ref{lem:generation-after-cover}, $W_R$ algebraically generates $A_R$.

The group $\Gamma$ preserves $\cG_R$. Lemma~\ref{lem:rank-two-intrinsic-normalizer} shows that every $\gamma\in\Gamma$ normalizes $A_R$ and acts on the open orbit by an affine automorphism whose finite-order linear part preserves $W_R$. Since $A_R^{\mathrm{unip}}$ is characteristic, this linear part descends to a finite-order automorphism of $S_R$ preserving the image of $W_R$. Lemma~\ref{lem:lattice} then implies that the scalar by which $\gamma$ acts on
\[
    H^0(R,\cO_R(K_{\cG_R}))\cong\det(W_R^\vee)
\]
has order at most $30$.

Choose $\gamma\in\Gamma$ mapping to a generator of $\Gal(Y/X)$. The pullback
\[
    \eta_R:=p^{[*]}\eta_Y\in H^0(R,\cO_R(K_{\cG_R}))
\]
is nonzero and hence spans this one-dimensional vector space. The action of $\gamma$ on $\eta_R$ has exact order $r$, because $\eta_Y$ has primitive deck character of order $r$. Therefore $r\le30$.
\end{proof}

\section{Rank two: purely transcendental case}\label{sec:case-e}

Throughout this section, $X$ is a $\mathbb Q$-factorial klt projective threefold over $\mathbb C$, and $\cF\subset TX$ is a rank two foliation with canonical singularities and $K_\cF\equiv0$. We treat the case where the algebraic part of $\cF$ has rank zero.

\begin{prop}\label{prop:rk2-alg0}
Let $X$ be a $\mathbb Q$-factorial klt projective threefold over $\mathbb C$, and let $\cF\subset TX$ be a rank two foliation with canonical singularities. Assume that $K_\cF\equiv0$ and that the algebraic part of $\cF$ has rank zero. Then there exists an integer $r$ with $1\le r\le30$ such that
\[
    rK_\cF\sim0.
\]
Equivalently, if $r$ is the exact torsion index of $K_\cF$, then $r\le30$.
\end{prop}

\begin{lem}[Source extraction in algebraic rank zero]\label{lem:case-e-source}
Under the hypotheses of Proposition~\ref{prop:rk2-alg0}, there exist:
\begin{enumerate}
    \item a connected normal projective klt threefold $Z$;
    \item a finite surjective quasi-\'etale morphism $g:Z\to X$;
    \item a connected commutative algebraic group $A_Z$ of dimension three acting on $Z$;
    \item a dense open orbit $U_Z\subset Z$;
    \item a two-dimensional subspace $W_Z\subset \Lie A_Z$;
\end{enumerate}
such that the inverse-image foliation
\[
    \cH_Z:=g^{-1}\cF
\]
is induced on $U_Z$ by $W_Z$ and has algebraic rank zero. Moreover, after choosing a point $z_0\in U_Z$, the orbit map
\[
    A_Z\longrightarrow U_Z,
    \qquad
    a\longmapsto a\cdot z_0,
\]
is an isomorphism, and
\[
    \cH_Z\cong\cO_Z^2.
\]
\end{lem}

\begin{proof}
Apply Theorem~\ref{thm:DO22-structure} to the codimension one foliation $(X,\cF)$. Alternative~(1) would give a finite quasi-\'etale cover $W'\times C\to X$ whose inverse-image foliation is induced by the projection to $C$; the fibers would yield positive-dimensional algebraic subvarieties tangent to $\cF$ through general points, contradicting algebraic rank zero. Hence alternative~(2) holds.

Thus there are normal projective klt varieties $W_1,W_2$, a finite quasi-\'etale cover
\[
    W_1\times W_2\to X,
\]
and a foliation $\cH_2$ on $W_2$ such that the inverse-image foliation is the pullback of $\cH_2$ by the projection to $W_2$. This inverse-image foliation has algebraic rank zero. Since the algebraic rank of such a product foliation is $\dim W_1+\rk_{\alg}\cH_2$, we get $\dim W_1=0$ and $\rk_{\alg}\cH_2=0$. Hence $\dim W_2=3$.

Set $Z:=W_2$, let $g:Z\to X$ be the induced finite quasi-\'etale morphism, and set $\cH_Z:=\cH_2$. By Theorem~\ref{thm:DO22-structure}(2), the variety $Z$ is an equivariant compactification of a connected commutative algebraic group $A_Z$ of dimension three, with dense open orbit $U_Z$, and $\cH_Z|_{U_Z}$ is induced by a codimension one Lie subgroup. Equivalently, there is a two-dimensional subspace $W_Z\subset\Lie A_Z$ inducing $\cH_Z$ on $U_Z$. The same theorem gives
\[
    \cH_Z\cong\cO_Z^{\dim Z-1}=\cO_Z^2.
\]
The algebraic rank of $\cH_Z$ is zero by the preceding paragraph.
\end{proof}

\begin{proof}[Proof of Proposition~\ref{prop:rk2-alg0}]
By the abundance theorem and Lemma~\ref{lem:R-to-Q-linear}, $K_\cF\sim_{\mathbb Q}0$. Let $r$ be the exact torsion index of $K_\cF$; we prove $r\le30$.

Let $\pi:Y\to X$ be the connected cyclic index one cover of $K_\cF$, set $\cG_Y:=\pi^{-1}\cF$, and let $\eta_Y$ be the tautological generator of $H^0(Y,\cO_Y(K_{\cG_Y}))$. By Lemma~\ref{lem:cyclic}, $\eta_Y$ has primitive deck character of order $r$.

Let $g:Z\to X$, $A_Z$, and $W_Z$ be as in Lemma~\ref{lem:case-e-source}. Let $R$ be the normalization of $X$ in the Galois closure of the compositum of $\mathbb C(Y)$ and $\mathbb C(Z)$ over $\mathbb C(X)$. Denote the induced maps by
\[
    \rho:R\to X,
    \qquad
    p:R\to Y,
    \qquad
    h:R\to Z,
\]
and set $\Gamma:=\Gal(R/X)$. As above, the covers $R\to X$, $p$, and $h$ are finite quasi-\'etale.

Apply Lemma~\ref{lem:finite-refinement-equivariant} to $h:R\to Z$. We obtain a connected commutative algebraic group $A_R$ with dense open orbit $U_R=h^{-1}(U_Z)$, a finite \'etale homomorphism
\[
    \phi_R:A_R\to A_Z,
\]
and an extension of the translation action of $A_R$ to $R$. Set
\[
    W_R:=(d\phi_R)^{-1}(W_Z)\subset\Lie A_R.
\]
Then $W_R$ is two-dimensional, and the inverse-image foliation $\cG_R:=\rho^{-1}\cF$ is induced on $U_R$ by $W_R$.

Choose a basis $w_1,w_2$ of $W_R$, and let $\widetilde w_1,\widetilde w_2$ be the corresponding global vector fields on $R$. As in the proof of Proposition~\ref{prop:rk2-alg1}, reflexivity gives $\widetilde w_i\in H^0(R,\cG_R)$, and the generically nonzero wedge $\widetilde w_1\wedge\widetilde w_2$ has an effective divisor $D$ satisfying
\[
    \cO_R(D)\cong\det\cG_R\cong\cO_R(-K_{\cG_R}).
\]
Since $K_{\cG_R}\equiv0$, the divisor $D$ is numerically trivial, hence zero. Therefore
\[
    W_R\otimes\cO_R\cong\cG_R,
    \qquad
    H^0(R,\cO_R(K_{\cG_R}))\cong\det(W_R^\vee).
\]

The foliation $\cH_Z$ has algebraic rank zero, so Lemma~\ref{lem:no-unipotent-direction} gives $W_Z\cap\Lie A_Z^{\mathrm{unip}}=0$. Lemma~\ref{lem:no-unipotent-after-cover} then gives
\[
    W_R\cap\Lie A_R^{\mathrm{unip}}=0.
\]
Thus $W_R$ injects into the Lie algebra of the semi-abelian quotient $S_R:=A_R/A_R^{\mathrm{unip}}$, whose dimension is at most three. Also, $W_Z$ algebraically generates $A_Z$ by Lemma~\ref{lem:rank-zero-generation}; hence $W_R$ algebraically generates $A_R$ by Lemma~\ref{lem:generation-after-cover}.

Since $\Gamma$ preserves $\cG_R$, Lemma~\ref{lem:rank-two-intrinsic-normalizer} applies. Thus every $\gamma\in\Gamma$ acts on $A_R$ through an affine automorphism whose finite-order linear part preserves $W_R$ and descends to $S_R$. By Lemma~\ref{lem:lattice}, the induced scalar on
\[
    H^0(R,\cO_R(K_{\cG_R}))\cong\det(W_R^\vee)
\]
has order at most $30$.

Choose $\gamma\in\Gamma$ mapping to a generator of $\Gal(Y/X)$. The nonzero pullback $p^{[*]}\eta_Y$ spans $H^0(R,\cO_R(K_{\cG_R}))$, and $\gamma$ acts on it with exact order $r$. Therefore $r\le30$.
\end{proof}

\section{Proof of the main theorems}\label{sec:proof-of-main-theorem}

\begin{proof}[Proof of Theorem~\ref{thm:main-pair}]
By Proposition~\ref{prop:reduction}, it suffices to prove the theorem when $X$ is $\mathbb Q$-factorial klt, $\dim X=3$, $\Ff$ is canonical, $B=0$, and
\[
    2\ge \rk\Ff>\rk_{\alg}\Ff\ge0.
\]
If $\rk\Ff=1$, the required bounded index is Proposition~\ref{prop:rank-one}. If $\rk\Ff=2$ and $\rk_{\alg}\Ff=1$, it is Proposition~\ref{prop:rk2-alg1}. If $\rk\Ff=2$ and $\rk_{\alg}\Ff=0$, it is Proposition~\ref{prop:rk2-alg0}. Thus the reduced case holds, and Proposition~\ref{prop:reduction} gives the theorem.
\end{proof}

\begin{proof}[Proof of Theorem~\ref{thm:main}]
This is the special case of Theorem~\ref{thm:main-pair} with $B=0$ and $\Ii=\{0\}$.
\end{proof}

\end{document}